\documentclass[11pt]{article}
\usepackage{amssymb,amsmath,amsfonts,amsthm,mathtools,bm,enumerate,color}
\usepackage{mathrsfs}
\usepackage[toc,page,title,titletoc,header]{appendix}
\usepackage{graphicx}

\usepackage{paralist}

\usepackage{tikz}
\usetikzlibrary{arrows,positioning,shapes.geometric}

\usepackage{algorithm,algorithmic}

\graphicspath{ {./figures/} }

\usepackage{indentfirst}
\usepackage{multicol}
\usepackage{booktabs}
\usepackage{url}
\usepackage[outdir=./]{epstopdf} 

\usepackage{hyperref}
\hypersetup{
    colorlinks=true, 
    linktoc=all,     
    linkcolor=blue,  
}
\numberwithin{equation}{section}

\newtheorem{Theorem}{Theorem}[section]
\newtheorem{Lemma}[Theorem]{Lemma}
\newtheorem{Proposition}[Theorem]{Proposition}
\newtheorem{Corollary}[Theorem]{Corollary}
\newtheorem{Assumption}{H.\!\!}

\theoremstyle{definition}

\theoremstyle{remark}
\newtheorem{Remark}{Remark}[section]

 \def\p{\partial} 
\def\to{\rightarrow}
     
\def\Om{\Omega}  \def\om{\omega} 
 
\newcommand{\q}{\quad}

\def\l{\label}    \def\fa{\forall}
\def\b{\beta}  \def\a{\alpha} 
 
\def\eps{\varepsilon}

 \def\t{\times}  
\def\ms{\medskip}

\def \la{\langle} \def\ra{\rangle}

\def\cF{\mathcal{F}}

\def\cH{\mathcal{H}}
\def\cI{\mathcal{I}}

\def\cP{\mathcal{P}}

\def\cS{\mathcal{S}}

\def\cU{\mathcal{U}}

\def\cW{\mathcal{W}}

\def\d{{\mathrm{d}}}

\def\sB{\mathbb{B}}

\def\sD{{\mathbb{D}}}
\def\sE{{\mathbb{E}}}
\def\sF{{\mathbb{F}}}
\def\sI{{\mathbb{I}}}

\def\sN{{\mathbb{N}}}
\def\sP{\mathbb{P}}

\def\sR{{\mathbb R}}

\DeclareMathOperator*{\esssup}{ess\,sup}

\newcommand{\lc}
{\mathrel{\raise2pt\hbox{${\mathop<\limits_{\raise1pt\hbox
{\mbox{$\sim$}}}}$}}}

\newcommand{\gc}
{\mathrel{\raise2pt\hbox{${\mathop>\limits_{\raise1pt\hbox{\mbox{$\sim$}}}}$}}}

\newcommand{\ec}
{\mathrel{\raise2pt\hbox{${\mathop=\limits_{\raise1pt\hbox{\mbox{$\sim$}}}}$}}}

\def\bb{\begin{equation}} \def\ee{\end{equation}}
\def\bbn{\begin{equation*}} \def\een{\end{equation*}}

\def\beqn{\begin{eqnarray}}  \def\eqn{\end{eqnarray}}

\def\beqnx{\begin{eqnarray*}} \def\eqnx{\end{eqnarray*}}

\def\bn{\begin{enumerate}} \def\en{\end{enumerate}}

\def\bd{\begin{description}} \def\ed{\end{description}}



\begin{document}

\title{
Path regularity of coupled McKean-Vlasov FBSDEs
}

\author{
Christoph Reisinger\thanks{
Mathematical Institute, University of Oxford, Oxford OX2 6GG, UK
 ({\tt christoph.reisinger@maths.ox.ac.uk, 
wolfgang.stockinger@maths.ox.ac.uk,
yufei.zhang@maths.ox.ac.uk})}
\and
Wolfgang Stockinger\footnotemark[1]
\and
Yufei Zhang\footnotemark[1]
}
\date{}

\maketitle

\noindent\textbf{Abstract.} 
This paper establishes H\"{o}lder time regularity  of solutions to  coupled McKean-Vlasov forward-backward stochastic differential equations (MV-FBSDEs).
This is not only of fundamental mathematical interest, but also essential for their numerical approximation.
We show that a solution triple to a MV-FBSDE  with Lipschitz coefficients is  $1/2$-H\"{o}lder continuous in time in the $L^{p}$-norm
provided that it admits a Lipschitz decoupling field. 
Special examples include
 decoupled MV-FBSDEs, 
coupled MV-FBSDEs with a small time horizon
 and 
 coupled stochastic Pontryagin systems 
 arsing from mean field control problems.

\medskip


\medskip



\medskip


\section{Introduction}\l{sec:intro}
In this paper, we establish path regularity of solutions to 
 fully-coupled McKean-Vlasov forward-backward stochastic differential equations (MV-FBSDEs).
Let 
$T>0$,
  $(\Om, \cF,  \sP)$ be
 a  complete probability space
 on which  a $d$-dimensional Brownian motion $W = (W^{(1)}_{t},\ldots, W^{(d)}_{t})_{t\in [0,T]}$ is defined,
 and  
 $\sF = (\cF_t)_{t\in [0,T]}$ be  the natural filtration of $W$
 augmented with an independent $\sigma$-algebra $\cF_0$.
We consider  the following MV-FBSDEs: 
 for $t \in [0,T]$,
\begin{subequations}\label{eq:mv-fbsde}
\begin{alignat}{2}
\d X_t&=b(t,X_t,Y_t,\sP_{(X_t,Y_t,Z_t)})\,\d t +\sigma (t,X_t,Y_t,  \sP_{(X_t,Y_t,Z_t)})\, \d W_t, 
\q 
&&X_0=\xi_0,
\l{eq:mv-fbsde_fwd}
\\
\d Y_t&=-f(t,X_t,Y_t,Z_t, \sP_{(X_t,Y_t,Z_t)})\,\d t+Z_t\,\d W_t,
\q 
&& Y_T=g(X_T,\sP_{X_T}),
\l{eq:mv-fbsde_bwd}
\end{alignat}
\end{subequations}
 where $\xi_0\in L^2(\cF_0;\sR^n)$, 
 $(b,\sigma, f,g)$ are given Lipschitz continuous functions,
$\sP_U$ denotes the law of a given random variable $U$,
 and
a solution triple  $(X,Y,Z)$
is an $\sR^n\t \sR^m\t \sR^{m\t d}$-valued
square-integrable 
adapted process satisfying \eqref{eq:mv-fbsde} $\sP$-almost surely.

Such equations  play an
important role in large population optimization problems 
(see e.g.~\cite{chassagneux2014,bensoussan2015,carmona2015,carmona2018a, gobet2019,lauriere2019}).
In particular, 
the solutions of \eqref{eq:mv-fbsde}
 give a stochastic representation of the solutions
and their derivatives 
 to certain nonlinear nonlocal 
 partial differential equations (PDEs) defined on the Wasserstein space,
 the so-called  nonlinear Feynman-Kac representation formula (see \cite{chassagneux2014,lauriere2019}).
Moreover,  by applying the stochastic maximum principle, we can construct both the equilibria of  mean field games and the solutions to  mean field control problems based on  solutions 
of \eqref{eq:mv-fbsde}.

Let $(X,Y,Z)$ be a given solution triple to \eqref{eq:mv-fbsde}. We aim to 
establish its  path regularity of the following form: for all $p\in \sN, t,s\in [0,T]$
and 
 for every
 partition $\pi=\{0=t_0<\cdots<t_N=T\}$,
\begin{align}
\sE\bigg[\sup_{0\le t\le T}|Z_t|^p\bigg]\le C_{(p,\xi_0)},
\q
\sE\left[\sup_{s\le r\le t}|X_r-X_s|^p\right]
 +\sE\left[\sup_{s\le r\le t}|Y_r-Y_s|^p\right]
& \le
C_{(p,\xi_0)}
|t-s|^{{p}/{2}},
\l{eq:XY_holder_intro}
\\
\sum_{i=0}^{N-1}
\sE\bigg[
\bigg(
\int_{t_i}^{t_{i+1}}
|Z_r-Z_{t_i}|^2\,\d r\bigg)^{p/2}
+
\bigg(
\int_{t_i}^{t_{i+1}}
|Z_r-Z_{t_{i+1}}|^2\,\d r\bigg)^{p/2}
\bigg]
& \le C_{(p,\xi_0)} 
|\pi|^{{p}/{2}},
\l{eq:Z_holder_intro}
\end{align}
where $|\pi|=\max_{i=0,\ldots, N-1}(t_{i+1}-t_i)$
and $C_{(p,\xi_0)}$ is a constant depending only on $p$ and the initial condition.
Such path regularity results are crucial for practical applications of MV-FBSDEs.
For example, 
an estimate of 
$(X,Y,Z)$ in the norm
$\sE\left[\sup_{0\le t\le T}|\cdot|^p\right]$ 
with a sufficiently large $p \in \sN$
is  essential for
establishing 
 convergence rates of the particle approximation
for \eqref{eq:mv-fbsde}
(so-called quantitative propagation of chaos results),
where   the marginal law $\sP_{(X_t,Y_t,Z_t)}$ is approximated 
by the empirical distribution of interacting particles at each $t\in [0,T]$
(see e.g.~\cite[Proposition 5]{lauriere2019}).
Moreover, 
by applying the stochastic maximum principle 
and studying the  path regularity of the associated  MV-FBSDEs
(i.e., the stochastic Pontryagin systems),
we can analyze the time regularity 
and discrete-time approximations
of   open-loop optimal controls
of (extended) mean field control problems via purely probabilistic arguments,
without analyzing the classical solutions to the associated 
infinite-dimensional PDEs on the measure spaces
(see \cite{reisinger2020_mfc}).
Finally, 
it is well-known that the H\"{o}lder regularity of the process $Z$ 
in \eqref{eq:Z_holder_intro}
plays a  crucial role in quantifying convergence rates of time-stepping schemes for 
 (MV-)FBSDEs,
 in particular the approximation error of the  process $Z$; 
see e.g.~\cite{zhang2004,lionnet2015} for  classical BSDEs
and \cite[Theorem 4.2]{reisinger2020} for  MV-FBSDEs.

Although such path regularity has been proved  for
classical FBSDEs (without mean field interaction) 
in various papers (e.g. \cite{bender2008,lionnet2015,zhang2017}),
{
to the best of our knowledge, there is no published work on the 
path regularity of solutions to MV-FBSDE \eqref{eq:mv-fbsde}
with 
possibly degenerate diffusion coefficient $\sigma$
and general Lipschitz continuous $f$,
even for the decoupled cases}.
In this work, we shall close the gap by showing 
that  a given solution triple to \eqref{eq:mv-fbsde} enjoys the regularity estimates
\eqref{eq:XY_holder_intro}-\eqref{eq:Z_holder_intro}
provided that the process $Y$ admits a Lipschitz decoupling field (see Theorem \ref{thm:mv-fbsde_regularity}).
Such condition holds if \eqref{eq:mv-fbsde} is uniquely solvable and 
stochastically stable with respect to the initial condition,
which can be verified 
for several practically important cases, including 
 decoupled MV-FBSDEs, 
coupled MV-FBSDEs with a small time horizon
(see Corollary \ref{cor:small_T})
 and coupled MV-FBSDEs whose coefficients satisfy a generalized monotonicity condition
 (see Corollary \ref{cor:mono}).

The mean field interaction and the strong coupling in \eqref{eq:mv-fbsde}
 pose a significant challenge for
establishing the path regularity
beyond those encountered in \cite{bender2008,lionnet2015,zhang2017}. 
Recall that  a crucial step in deriving these path regularity results for classical 
decoupled FBSDEs is to represent 
 the  Malliavin derivatives of the solutions
by using 
the first variation processes
(i.e., the derivatives of the solutions  with respect to the initial condition)
and their inverse.
However, such a representation no longer holds for solutions to MV-FBSDEs, 
since equations for the first variation processes
 will involve the derivatives of 
 marginal distributions of the solutions
with respect to  the initial condition,
 which do not appear in equations for the Malliavin derivatives of the solutions.
 Moreover, due to the strong coupling between the forward and backward equations,
 the first variation of the forward component of the solution
 will depend on 
 the first variations of the backward components,
 which creates an essential difficulty in establishing the invertibility of the first variation processes.
 
 We shall overcome the above difficulties by employing the  
   decoupling field of the solution,
 which enables us to express the backward component $Y$ of the solution  as
a  function of the forward component $X$
and then 
rewrite the  coupled MV-FBSDE as a  
 decoupled FBSDE,
whose coefficients depend on the decoupling field and the flow $(\sP_{(X_t,Y_t,Z_t)})_{t\in [0,T]}$.
Note that these modified  coefficients are in general  merely
Lipschitz continuous in space and
 square integrable in time, 
 due to the lack of time regularity 
 of the decoupling field and the flow $t\mapsto \sP_{Z_t}$.
  We 
  then establish a representation formula of the process $Z$,
  based on ``partial" first variation processes of the solutions
  and  the weak derivatives of the decoupling field and 
   these  irregular coefficients,
   which subsequently leads us to the desired path regularity of \eqref{eq:mv-fbsde};
see the discussion at the end of Section  \ref{sec:proof_regularity}
for details.

We state the path regularity
results for   coupled MV-FBSDE \eqref{eq:mv-fbsde} in Section \ref{sec:mv-fbsde}
and present its proof in Section \ref{sec:proof_regularity}. 
 Appendix \ref{sec:proof_mono} is  devoted to the proofs of some technical results.
 
\textbf{Notation.}
We end this section by introducing  some  notation  used throughout  this paper.
%
%
%
For any given $n\in \sN$ and $x\in \sR^n$, we  denote by 
  $\sI_n$  the $n\t n$ identity matrix,
  by ${0}_{n}$ the  zero element
  of  $\sR^{n}$
   and
by
 $\bm{\delta}_{x}$
 the Dirac measure supported at $x$.
We shall denote by $\la \cdot,\cdot\ra$
the usual inner product in a given Euclidean space
and by   $|\cdot|$ the norm induced by $\la \cdot,\cdot\ra$,
which in particular satisfy  for all 
$n,m,d\in \sN$
and
$\theta_1=(x_1,y_1,z_1),\theta_2=(x_2,y_2,z_2)\in \sR^n\t \sR^m\t \sR^{m\t d}$
that
$\la z_1,z_2\ra =\textrm{trace}(z^*_1z_2)$
and 
$\la \theta_1,\theta_2\ra =\la x_1,x_2\ra+\la y_1,y_2\ra+\la z_1,z_2\ra$,
where  $(\cdot)^*$ denotes the   transposition  of a matrix.

We then introduce several spaces:
for each $p \ge 1$, $k \in \sN$, $t\in [0,T]$
and Euclidean space $(E,|\cdot|)$,
$L^p(\Om; E)$ is the space  of 
 $E$-valued
$\cF$-measurable
random variables $X$ satisfying
$\|X\|_{L^p}=\sE[|X|^p]^{1/p}<\infty$,
and 
$L^p(\cF_t; E)$ is the subspace  of $L^p( \Om;E)$
containing all 
$\cF_t$-measurable
random variables;
$\cS^p(t,T;E)$ is the space of 
$\sF$-progressively  measurable 
processes
$Y: \Om\t [t,T]\to E$ 
satisfying $\|Y\|_{\cS^p}=\sE[\esssup_{s\in [t,T]}|Y_s|^p]^{1/p}<\infty$,
and
$\cS^\infty$ is the subspace of  $\cS^p(t,T;E)$ 
containing all uniformly bounded processes
$Y$ 
satisfying $\|Y\|_{\cS^\infty}=\esssup_{(s,\om)}|Y_s|<\infty$;
 $\cH^p(t,T; E)$ is the space of 
  $\sF$-progressively measurable
 processes 
$Z: \Om\t [t,T]\to E$  
 satisfying $\|Z\|_{\cH^p}=\sE[(\int_t^T|Z_s|^2\,\d s)^{p/2}]^{1/p}<\infty$;
 $\sD^{1,p}(E)$ is the  space of 
 Malliavin differentiable random variables.
 For notational simplicity, 
 when $t=0$,
 we often
denote 
  $\cS^p=\cS^p(0,T;E)$
  and $\cH^p=\cH^p(0,T;E)$,
if no confusion occurs.

Moreover, for every  
Euclidean space $(E, |\cdot|)$,
we denote by $\cP_2(E)$  the metric space of  probability measures 
$\mu$
on $E$ satisfying $\|\mu\|_2=(\int_E |x|^2\,\d \mu(x))^{1/2}<\infty$,
endowed with the  $2$-Wasserstein metric defined by 
$$
\cW_2(\mu_1,\mu_2)
\coloneqq \inf_{\kappa\in \Pi(\mu_1,\mu_2)} \left(\int_{E\t E}|x-y|^2\d \kappa( x, y)\right)^{1/2},
\q \mu_1,\mu_2\in \cP_2(E),
$$
where $\Pi(\mu_1,\mu_2)$ is the set of all couplings of $\mu_1$ and $\mu_2$, i.e.,
$\kappa\in \Pi(\mu_1,\mu_2)$ is a probability measure on $E\t E$ such that $\kappa(\cdot\t E)=\mu_1$ 
and $\kappa(E\t \cdot)=\mu_2$.

\section{Path regularity of fully coupled MV-FBSDEs}\l{sec:mv-fbsde}
In this section, we establish an $L^p$-path regularity result for \eqref{eq:mv-fbsde},
whose coefficients 
$(b, \sigma, f, g)$  satisfy the following standing assumptions:
\begin{Assumption}\l{assum:mv-fbsde_lip}
Let $n,m,d\in \sN$,
$T\in [0,\infty)$
and let
$b: [0,T]\t \sR^n\t \sR^m \t \cP_2(\sR^{n}\t \sR^m\t \sR^{m\t d})\to \sR^n $,
$\sigma: [0,T] \t \sR^n\t \sR^m\t \cP_2(\sR^{n}\t \sR^m\t \sR^{m\t d})\to \sR^{n\t d} $,
$f: [0,T]\t \sR^n\t \sR^m\t \sR^{m\t d}\t \cP_2(\sR^{n}\t \sR^m\t \sR^{m\t d})\to \sR^{m} $
and 
$g: \sR^n\t \cP_2(\sR^{n})\to \sR^{m}$
be measurable functions
satisfying 
for some $L,K\in [0,\infty)$
   that:
\begin{enumerate}[(1)]
\item\l{item:mv-fbsde_lip}
For all $t\in [0,T]$,
the functions
$b(t,\cdot)$, $\sigma(t,\cdot)$, $f(t,\cdot)$ and $g(\cdot)$
are uniformly Lipschitz continuous in all  variables
with a Lipschitz constant $L$.
\item  \l{item:mv-fbsde_moment}
$
 \|b(\cdot,{0},{0},\bm{\delta}_{{0}_{n+m+md}})\|_{L^2(0,T)}
 +
 \|f(\cdot,{0},{0}, {0},\bm{\delta}_{{0}_{n+m+md}})\|_{L^2(0,T)}+
 |g({0},\bm{\delta}_{{0}_{n}})|\le K
 $,
and it holds for all $\mu\in \cP_2(\sR^{m\t d})$ that
$\|\sigma(\cdot,0,0,\bm{\delta}_{0_{n+m}}\t \mu)\|_{L^\infty(0,T)}\le K$.

\end{enumerate}
\end{Assumption}
\begin{Remark}
Throughout this paper,
we shall  denote by $C\in [0,\infty)$ a generic constant, which is independent of the initial condition $\xi_0$,
though it may depend on the constants appearing in the assumptions  
  and may take a different value at each occurrence. 
Dependence of $C$ on additional parameters will be indicated explicitly by $C_{(\cdot)}$, e.g. $C_{(p)}$ for some $p\in \sN$.
\end{Remark}

The following theorem presents a general path regularity result for 
a given solution to 
 \eqref{eq:mv-fbsde},
provided that 
 the solution admits a  decoupling field and
 enjoys a natural
  moment estimate.
In the subsequent analysis,
for a given triple $(X,Y,Z) \in  \cS^2(\sR^n) \t \cS^2(\sR^m) \t \cH^2(\sR^{m\t d})$ 
and a constant $L_v\in [0,\infty)$,
we say $Y$ admits an $L_v$-Lipschitz decoupling field if
there
exists a measurable function $v :[0,T ] \t \sR^n\to \sR^m$ 
such that
$\sP(\fa t\in [0,T], Y_t=v(t,X_t))=1$
and
it holds for all $t\in [0,T]$ 
and $x,x'\in \sR^n$
that
$|v(t,x)-v(t,x')|\le L_v|x-x'|$.
For the sake of readability,  we postpone the  proof 
of Theorem \ref{thm:mv-fbsde_regularity}
to Section \ref{sec:proof_regularity}.

\begin{Theorem}\l{thm:mv-fbsde_regularity}
Suppose (H.\ref{assum:mv-fbsde_lip}) holds.
Let 
 $\xi_0\in L^2(\cF_0;\sR^n)$,
$L_v,M\in [0,\infty)$
and 
$(X,Y,Z) \in  \cS^2(\sR^n) \t \cS^2(\sR^m) \t \cH^2(\sR^{m\t d})$ 
be a solution to   \eqref{eq:mv-fbsde}
satisfying the following two conditions:
(1) $Y$ admits an $L_v$-Lipschitz decoupling field;
(2) 
$\|X\|_{\cS^2}+\|Y\|_{\cS^2}+\|Z\|_{\cH^2}\le M(1+\|\xi_0\|_{L^2})$.
 Then 
we have that:
\begin{enumerate}[(1)]
\item\l{item:mv_Z_bdd}
There exists a constant $C>0$ such that  
   for $\d \sP\otimes \d t$ a.e., $|Z_t|\le C| {\sigma}(t,X_t,Y_t,\sP_{(X_t,Y_t,Z_t)})|$.
Consequently, 
   for all  $p\ge 2$,
there exists
a constant $C_{(p)}>0$
such that
$
\|X\|_{\cS^p}+\|Y\|_{\cS^p}+ \|Z\|_{\cS^p}
\le 
 C_{(p)}
\big(1+
\|\xi_0\|_{L^{p}}
\big)$.
\item\l{item:mv_XY_Holder}
For any $p\ge 2$, 
there exists 
a constant $C_{(p)}>0$
such that 
 it holds for all $0\le s\le t\le T$ that
 $ \sE\left[\sup_{s\le r\le t}|X_r-X_s|^p\right]^{1/p}
 +\sE\left[\sup_{s\le r\le t}|Y_r-Y_s|^p\right]^{1/p}
 \le
C_{(p)}(1+\|\xi_0\|_{L^{p}})
|t-s|^{{1}/{2}}
$.
\item \l{item:mv_Z_Holder}
Assume further that 
for each $(t,x,y)\in [0,T]\t\sR^n\t \sR^m$,
the function
 $\cP_2(\sR^{n}\t \sR^m\t \sR^{m\t d})\ni \eta\mapsto 
 \sigma(t,x,y,\eta)\in\sR^{n\t d}$
 depends only on  
 the   marginal $\pi_{1,2}\sharp \eta= \eta(\cdot\t \sR^{m\t d})$
 of the measure $\eta$,
and 
there exists a constant $C_\sigma\in [0,\infty)$
such that 
 for all $s,t \in [0,T]$, $(x,y,\eta)\in \sR^n\t \sR^m\t \cP_2(\sR^{n}\t \sR^m\t \sR^{m\t d})$,
\bb\l{eq:sigma_local_holder}
|\sigma(s,x,y,\eta)-{\sigma}(t,x,y,\eta)|\le 
C_\sigma
\big(1+|x|+|y|+
\|\pi_{1,2}\sharp \eta\|_2
\big)|s-t|^{1/2}.
\ee
Then 
for any $p\ge 2$ and $\eps>0$, 
there exists 
a constant $C_{(p,\eps)} >0$
such that it holds
 for every
 partition $\pi=\{0=t_0<\cdots<t_N=T\}$
with  stepsize $|\pi|=\max_{i=0,\ldots, N-1}(t_{i+1}-t_i)$ that,
\begin{align}\l{eq:Z_holder}
\begin{split}
&\sum_{i=0}^{N-1}
\sE\bigg[
\bigg(
\int_{t_i}^{t_{i+1}}
|Z_r-Z_{t_i}|^2\,\d r\bigg)^{p/2}
+
\bigg(
\int_{t_i}^{t_{i+1}}
|Z_r-Z_{t_{i+1}}|^2\,\d r\bigg)^{p/2}
\bigg]^{1/p}
\\
&\quad
 \le C_{(p,\eps)} 
(1+\|\xi_0\|_{L^{p+\eps}})|\pi|^{{1}/{2}}.
\end{split}
\end{align}
\end{enumerate}


\end{Theorem}

\begin{Remark}\l{eq:rmk_regularity}

Note that the
 H\"{o}lder regularity of the processes $X$, $Y$ in Item 
 (\ref{item:mv_XY_Holder})
has the optimal dependence on the integrability of the initial condition $\xi_0$ and the time regularity of the coefficients $b$, $\sigma$ and $f$.
The dependence 
on $\|\xi_0\|_{L^{p+\eps}}$ in \eqref{eq:Z_holder} 
appears due to the application of H\"{o}lder's inequality in the analysis,
which is sharp for deterministic initial data since $\|x_0\|_{L^{p}}=|x_0|$ for all $p\ge 2$.
Even though we  only assume 
that  
$\xi_0\in L^2(\cF_0;\sR^n)$,
the moment bounds and regularity estimates 
in  Theorem \ref{thm:mv-fbsde_regularity} obviously hold 
if $\xi_0\not\in L^p(\cF_0;\sR^n)$ for some $p> 2$,
since the right-hand side would be infinity.

\end{Remark}

%
%

The conditions in Theorem \ref{thm:mv-fbsde_regularity}
are satisfied by most MV-FBSDEs appearing in practice.
In particular, 
the following proposition shows that 
the solution to \eqref{eq:mv-fbsde} admits  a Lipschitz decoupling field 
if \eqref{eq:mv-fbsde} is uniquely solvable and stochastically stable,
whose proof follows from similar 
arguments as that of 
\cite[Proposition 5.7]{carmona2015}.

\begin{Proposition}\l{eq:decoupling_exist}
Assume 
for all  
$t\in [0,T]$ and 
 $\xi\in L^2(\cF_t;\sR^n)$ that 
there exists a unique triple of processes
$(X^{t,\xi},Y^{t,\xi},Z^{t,\xi}) \in  \cS^2(t,T;\sR^n) \t \cS^2(t,T;\sR^m) \t \cH^2(t,T;\sR^{m\t d})$ 
satisfying \eqref{eq:mv-fbsde} on $[t,T]$ with  
the initial condition $X^{t,\xi}_t=\xi$.
Assume further that there exists a constant $\bar{L}>0$ such that 
it holds for all $t\in [0,T]$  and $\xi,\xi'\in L^2(\cF_t;\sR^n)$ that 
$\|Y^{t,\xi}_t-Y^{t,\xi'}_t\|_{L^2}\le \bar{L}\|\xi-\xi'\|_{L^2}$.
Then for all $\xi_0\in L^2(\cF_0;\sR^n)$,  \eqref{eq:mv-fbsde} admits a unique solution 
 $(X,Y,Z)\in \cS^2(\sR^n) \t \cS^2(\sR^m) \t \cH^2(\sR^{m\t d})$
and the process $Y$ admits an $\bar{L}$-Lipschitz decoupling field.
\end{Proposition}
\begin{Remark}\l{rmk:decoupling_Z}
Note that due to the mean field interaction in \eqref{eq:mv-fbsde},
the   decoupling field $v$ depends on the law of the initial condition $\xi_0$.
Hence unlike
for the classical FBSDEs, in general
we do not have  for all $(t,x)\in [0,T]\t \sR^n$ that 
$Y^{t,x}_t=v(t,x)$.
This creates   a significant challenge
in establishing the  H\"{o}lder regularity of the mapping
$t\mapsto v(t,x)$ for a given $x\in \sR^d$.
In fact, 
a common approach in the existing literature
to 
analyze the 
time regularity of $v$ 
usually
involves establishing the relation 
$v(t,\cdot)=\cU(t,\cdot,\sP_{X_t})$ for all $t\in [0,T]$,
 identifying the map $\cU$ as a solution to
  an  infinite-dimensional PDE on $[0,T]\t \sR^n\t \cP_2(\sR^n)$,
and then analyzing this PDE
under strong regularity assumptions on the coefficients of \eqref{eq:mv-fbsde}, such as the
boundedness and
high-order differentiability conditions 
(see e.g.~\cite{chassagneux2014}).

\end{Remark}

We now present two concrete structural assumptions for the coefficients of \eqref{eq:mv-fbsde}
under which 
\eqref{eq:mv-fbsde} is uniquely solvable,
 stochastically stable and the solution enjoys a  natural moment estimate.
 The first one shows that 
Theorem \ref{thm:mv-fbsde_regularity} holds if 
the terminal time $T$ is sufficiently small compared to 
the coupling between \eqref{eq:mv-fbsde_fwd} and  \eqref{eq:mv-fbsde_bwd}
(see e.g.~\cite{delarue2002,chassagneux2014,gobet2019}),
which includes  decoupled MV-FBSDEs as  special cases.
Similar results can be extended to 
 weakly coupled (MV-)FBSDEs  as in  \cite{bender2008},
 where the function $b$ is strongly decreasing in $x$ 
 or the function $f$ is strongly decreasing in $y$.

 \begin{Corollary}\l{cor:small_T}
 Suppose (H.\ref{assum:mv-fbsde_lip}) holds
 and let the functions 
 $\sigma$ and $g$ 
 satisfy for some 
 $   L^g,L^\sigma_z\in [0,\infty)$,
  for all $(t,x,y, Y)\in [0,T]\t \sR^n\t \sR^m\t L^2(\Om; \sR^m)$,
$X,X'\in L^2(\Om; \sR^n)$, 
$Z,Z'\in L^2(\Om; \sR^{m\t d})$
  that
$|\sigma(t,x,y,\sP_{(X,Y,Z)})-\sigma(t,x,y,\sP_{(X,Y,Z')})|\le L^\sigma_z \cW_2(\sP_Z,\sP_{Z'})$
and
 $\|g(X,\sP_{X})-g(X',\sP_{X'})\|_{L^2}\le L^g\|X-X'\|_{L^2}$. 
 If 
 $c_0\coloneqq L^\sigma_zL^g<1$,
then there exists a constant $C_{(L,c_0)}>0$
such that it holds
for all $T\le C_{(L,c_0)}$ and  $\xi\in L^2(\cF_0;\sR^n)$ that
\eqref{eq:mv-fbsde} admits a unique solution 
$(X,Y,Z) \in  \cS^2(\sR^n) \t \cS^2(\sR^m) \t \cH^2(\sR^{m\t d})$ 
satisfying the regularity estimates in Theorem \ref{thm:mv-fbsde_regularity},
Items (\ref{item:mv_Z_bdd})--(\ref{item:mv_Z_Holder}).

In particular,  
if \eqref{eq:mv-fbsde} is decoupled in the sense that
\eqref{eq:mv-fbsde_fwd}
depends only on the process $X$ and the flow
$(\sP_{X_t})_{t\in[0,T]}$,
then 
the regularity estimates in  Theorem \ref{thm:mv-fbsde_regularity},
Items (\ref{item:mv_Z_bdd})--(\ref{item:mv_Z_Holder})
hold for all  $T\in [0,\infty)$.

\end{Corollary} 
\begin{proof}
One can establish the desired properties in Proposition \ref{eq:decoupling_exist}
by extending 
 the fixed-point arguments in 
\cite[Theorem 8.2.1 and Corollary 8.2.2]{zhang2017}
   for classical FBSDEs 
 to the present setting 
with mean field interaction,
whose detailed steps are 
  omitted.
  \end{proof}

We now  show that 
Theorem \ref{thm:mv-fbsde_regularity} holds
for  coupled MV-FBSDEs with  an arbitrary terminal time,
provided that the coefficients satisfy a
 monotonicity condition.

\begin{Corollary}\l{cor:mono}
 Suppose (H.\ref{assum:mv-fbsde_lip}) holds. 
 Assume further the  monotonicity condition holds: 
there exist
 $\a_1,\b_1,\b_2, L_\phi\in[ 0,\infty)$,
 $G\in \sR^{m\t n}$
 and
 measurable functions
$\phi_1:L^2(\Om;\sR^n)^{\otimes 2}\to [0,\infty)$,
$\phi_2:[0,T]\t L^2(\Om;\sR^n\t \sR^m\t\sR^{m\t d})^{\otimes 2}\to [0,\infty)$
such that 
for all 
$t\in [0,T]$, $i\in \{1,2\}$, 
 $\Theta_i\coloneqq (X_i,Y_i,Z_i)\in L^2(\Om; \sR^n\t \sR^m\t \sR^{m\t d})$,
\begin{align*}
\begin{split}
&\sE[\la b(t,X_1,Y_1,\sP_{\Theta_1})-b(t,X_2,Y_2,\sP_{\Theta_2}), G^*( Y_1-Y_2)\ra]
 \\
&\quad 
+\sE[\la \sigma(t,X_1,Y_1,\sP_{\Theta_1})-\sigma(t,X_2,Y_2,\sP_{\Theta_2}), G^*(  Z_1-Z_2)\ra]
\\
&\quad 
+ \sE[\la -f(t,\Theta_1,\sP_{\Theta_1})+f(t,\Theta_2,\sP_{\Theta_2}), G (X_1-X_2)\ra] 
\\
&\quad
 \le -\beta_1
 \phi_1(X_1,X_2)
 -\beta_2
  \phi_2(t,\Theta_1, \Theta_2),
\\
&\sE[\la g(X_1,\sP_{X_1})-g(X_2,\sP_{X_2}), G( X_1-X_2)\ra]
\ge  \a_1 \phi_1(X_1,X_2).
\end{split}
\end{align*}
Moreover, 
one of the following two conditions is satisfied:
(1) $\b_2>0$ and for all 
$t\in [0,T]$,
\begin{align}\l{eq:b_sigma_lip}
\begin{split}
\| (b,\sigma)(t,X_1,Y_1,\sP_{\Theta_1})-(b,\sigma)(t,X_2,Y_2,\sP_{\Theta_2})\|^2_{L^2}
\le  L_\phi(\|X_1-X_2\|_{L^2}^2+ \phi_2(t,\Theta_1, \Theta_2));
\end{split}
\end{align}
or (2) $\a_1,\b_1>0$ and for all 
$t\in [0,T]$  that 
\begin{align}\l{eq:f_g_lip}
\begin{split}
&\| f(t,X_1,Y_2,Z_2,\sP_{(X_1,Y_2,Z_2)})-f(t,X_2,Y_2,Z_2,\sP_{(X_2,Y_2,Z_2)})\|^2_{L^2}
\\
&\quad
+\|g(X_1,\sP_{X_1})-g(X_2,\sP_{X_2})\|^2_{L^2}
\le  L_\phi \phi_1(X_1, X_2).
\end{split}
\end{align}
Then
it holds
for all $\xi_0\in L^2(\cF_0;\sR^n)$ that
\eqref{eq:mv-fbsde} admits a unique solution 
$(X,Y,Z) \in  \cS^2(\sR^n) \t \cS^2(\sR^m) \t \cH^2(\sR^{m\t d})$ 
satisfying the regularity estimates  in Theorem \ref{thm:mv-fbsde_regularity},
Items (\ref{item:mv_Z_bdd})--(\ref{item:mv_Z_Holder}).

\end{Corollary}
\begin{proof}
The desired 
unique solvability, stochastic stability  and the moment estimates
follow from  a stability analysis of \eqref{eq:mv-fbsde}
 under the  monotonicity condition
and the 
 continuation method as in \cite{peng1999,bensoussan2015},
whose detailed steps 
can be found in Appendix \ref{sec:proof_mono}.
%
\end{proof}
\begin{Remark}

The monotonicity condition in Corollary \ref{cor:mono}
 is a natural generalization of
the well-known $G$-monotonicity condition in the existing literature
(see (H2.2) in \cite{peng1999} 
for FBSDEs
or 
 Assumption (A.1) in \cite{bensoussan2015}
for  MV-FBSDEs), which corresponds to the case
 where   $G\in \sR^{m\t n}$ is  a full-rank matrix, 
$\phi_1(X_1,X_2)=\|G(X_1-X_2)\|^2_{L^2}$ and 
 $\phi_2(t,\Theta_1,\Theta_2)=\|G^*(Y_1-Y_2)\|^2_{L^2}+\|G^*(Z_1-Z_2)\|^2_{L^2}$.
  
 More importantly, the generalized monotonicity condition can be applied to many  FBSDEs arising from 
 control problems 
whose coefficients enjoy specific structural conditions but fail to satisfy  the $G$-monotonicity condition.
For example, 
one can consider MV-FBSDEs with $n=m$, $b(t,X,Y,\sP_{(X,Y)})=b(t,X,\hat{\a}(t,X,Y,\sP_{(X,Y)}))$
and $\sigma(t,X,Y,\sP_{(X,Y)})=\sigma(t,X,\sP_{X})$
arising from applying
the stochastic maximum principle
to  extended mean field control problems
(see e.g.~\cite{carmona2015,reisinger2020_mfc}).
In this case, 
the coefficients in general do not satisfy 
the $G$-monotonicity condition by virtue of the non-monotonicity of the function $\hat{\a}$.
However,
by choosing $G=\sI_n$ and $\phi_2=\|\hat{\a}(t,X_1,Y_1,\sP_{(X_1,Y_1)})-\hat{\a}(t,X_2,Y_2,\sP_{(X_2,Y_2)})\|^2_{L^2}$,
the generalized monotonicity condition 
can still be satisfied 
under natural convexity conditions;
see \cite[Proposition 3.3]{reisinger2020_mfc}
for details.
\end{Remark}

\section{Proof of Theorem \ref{thm:mv-fbsde_regularity}}\l{sec:proof_regularity}

In this section, we prove the $L^p$-path regularity results given in Section \ref{sec:mv-fbsde}.
Throughout this proof,
let $(X,Y,Z) \in  \cS^2(\sR^n) \t \cS^2(\sR^m) \t \cH^2(\sR^{m\t d})$ 
be a given solution to   \eqref{eq:mv-fbsde}
with  the  decoupling field 
$v :[0,T ] \t \sR^n\to \sR^m$.
For notational simplicity, we define the following functions 
$\tilde{b}: [0,T]\t  \sR^n\to \sR^n$,
$\tilde{\sigma}: [0,T]\t  \sR^n\to \sR^{n\t d}$,
$\tilde{f}: [0,T]\t  \sR^n\t \sR^m\t \sR^{m\t d}\to \sR^m$
and $\tilde{g}:   \sR^n\to \sR^{m}$: 
for all $(t,x,y,z)\in [0,T] \t \sR^n\t \sR^m\t \sR^{m\t d}$,
\begin{alignat*}{2}
\tilde{b}(t,x)&\coloneqq b(t,x,v(t,x),\sP_{(X_t,Y_t,Z_t)}), \q 
&&\tilde{\sigma}(t,x)\coloneqq \sigma(t,x,v(t,x),\sP_{(X_t,Y_t,Z_t)}),
\\
\tilde{f}(t,x,y,z)&\coloneqq f(t,x,y,z,\sP_{(X_t,Y_t,Z_t)}),
\q 
&&\tilde{g}(x)\coloneqq g(x,\sP_{X_T}).
\end{alignat*}
Then it is clear that  $(X,Y,Z) $
satisfies the following FBSDE:
\begin{subequations}\label{eq:fbsde}
\begin{align}
\d X_t&=\tilde{b}(t,X_t)\,\d t +\tilde{\sigma} (t,X_t)\, \d W_t, 
\q X_0=\xi_0,
\label{eq:fbsde_fwd}
\\
\d Y_t&=-\tilde{f}(t,X_t,Y_t,Z_t)\,\d t+Z_t\,\d W_t,
\q
 Y_T=\tilde{g}(X_T).
 \label{eq:fbsde_bwd}
\end{align}
\end{subequations}

The following lemma presents several regularity properties of the functions $(\tilde{b},\tilde{\sigma},\tilde{f},\tilde{g})$.
Note that these functions are in general \textit{discontinuous} in time,   due to the lack of time regularity of the decoupling field $v$ and the flow 
$t\mapsto \sP_{Z_t}$.

\begin{Lemma}\l{lemma:lipschitz_coefficients}
 Assume the setting in the Theorem \ref{thm:mv-fbsde_regularity}.
Then 
there exists a constant $C$ such that
$(\tilde{b},\tilde{\sigma},\tilde{f},\tilde{g})$ are 
 measurable functions
 which are 
$C$-Lipschitz continuous with respect to the spatial variables 
uniformly with respect to $t$,
and satisfy that
$
\|\tilde{b}(\cdot, 0)\|_{L^2(0,T)}
+\|\tilde{\sigma}(\cdot, 0)\|_{L^\infty(0,T)}+
+\|\tilde{f}(\cdot,0,0,0)\|_{L^2(0,T)}
+|\tilde{g}(0)|
\le 
C(1+\|\xi_0\|_{L^2})
$.
\end{Lemma}

\begin{proof}

Note that \cite[Lemma 2.2]{bandini2016} shows that 
the mapping $[0,T]\ni t\mapsto {\sP_{(X_t,Y_t,Z_t)}}\in \cP_2(\sR^{n}\t \sR^m\t \sR^{m\t d})$ is measurable 
if  for all continuous function $\phi:\sR^n\t \sR^m\t \sR^{m\t d}\to \sR$ with quadratic growth,
the map $[0,T]\ni t\mapsto \sE[\phi(X_t,Y_t,Z_t)]\in \sR$  is measurable,
which holds 
in the present setting
due to the fact that $(X,Y,Z)\in  \cS^2(\sR^n) \t \cS^2(\sR^m) \t \cH^2(\sR^{m\t d})$ and Fubini's theorem.
Then, we can deduce
from the measurability of the functions $(b, \sigma, f, g)$
 that 
the functions $(\tilde{b},\tilde{\sigma},\tilde{f},\tilde{g})$ are measurable.
The $L_v$-Lipschitz continuity of the decoupling field $v$ and the $L$-Lipschitz continuity of the functions $(b, \sigma, f, g)$ in (H.\ref{assum:mv-fbsde_lip}) imply that 
the functions
$(\tilde{b},\tilde{\sigma},\tilde{f},\tilde{g})$  are  Lipschitz continuous with respect to their spatial variables, 
uniformly with respect to $t$.
Hence, it remains to show the integrability condition. 
Note that the Lipschitz continuity of the decoupling field $v$ and H\"{o}lder's inequality imply  that
\begin{align}\l{eq:v_bdd}
\begin{split}
\sup_{t\in [0,T]}|v(t,0)|&\le \sup_{t\in [0,T]}(\sE[|Y_t|]+\sE[|v(t,0)-v(t,X_t)|])
\\
&
\le \|Y\|_{\cS^2}+L_v\|X\|_{\cS^2}
\le C(1+\|\xi_0\|_{L^2}).
\end{split}
\end{align}
Thus we can obtain
for a.e.~$t\in [0,T]$   that 
\begin{align}\l{eq:bf_0}
\begin{split}
&|\tilde{b}(t, 0)|
+|\tilde{f}(t,0,0,0)|
=
|b(t,0,v(t,0),\sP_{(X_t,Y_t,Z_t)})|
+|{f}(t,0,v(t,0),0,\sP_{(X_t,Y_t,Z_t)})|
\\
&\le
|b(t,0,0,\bm{\delta}_{{0}_{n+m+md}})|
+|{f}(t,0,0,0,\bm{\delta}_{{0}_{n+m+md}})|
+C(|v(t,0)|+\cW_2(\sP_{(X_t,Y_t,Z_t)},\bm{\delta}_{{0}_{n+m+md}}))
\\
&\le
C(
|b(t,0,0,\bm{\delta}_{{0}_{n+m+md}})|
+|{f}(t,0,0,0,\bm{\delta}_{{0}_{n+m+md}})|
+
1+\|\xi_0\|_{L^2}+\|(X_t,Y_t,Z_t)\|_{L^2}),
\end{split}
\end{align}
which together with the assumption that 
$\|X\|_{\cS^2}+\|Y\|_{\cS^2}+\|Z\|_{\cH^2}\le M(1+\|\xi_0\|_{L^2})$
 gives us that 
$$
\|\tilde{b}(\cdot, 0)\|_{L^2(0,T)}
+\|\tilde{f}(\cdot,0,0,0)\|_{L^2(0,T)}
\le 
C(1+\|\xi_0\|_{L^2}+\|(X,Y,Z)\|_{\cH^2})
\le 
C(1+\|\xi_0\|_{L^2}).
$$
Similarly, 
by setting $\bm{0}_{n+m}$ to be the $\sR^{n}\t \sR^m$-valued zero random variable,
we have for a.e.~$t\in[0,T]$ that
$\sP_{(\bm{0}_{n+m},Z_t)}=\bm{\delta}_{{0}_{n+m}}\t \sP_{Z_t}$
and hence that
\begin{align*}
&|\tilde{\sigma}(t, 0)|
=|\sigma(t,0,v(t,0),\sP_{(X_t,Y_t,Z_t)})|
\\
&\le
|\sigma(t,0,0, \sP_{(\bm{0}_{n+m},Z_t)})|
+C(|v(t,0)|+\cW_2(\sP_{(X_t,Y_t,Z_t)},\sP_{(\bm{0}_{n+m},Z_t)}))
\\
&\le
C(1+\|\xi_0\|_{L^2}+\|(X_t,Y_t)\|_{L^2}),
\end{align*}
which implies that $\|\tilde{\sigma}(\cdot, 0)\|_{L^\infty(0,T)}\le C(1+\|\xi_0\|_{L^2}+\|(X,Y)\|_{\cS^2})\le C(1+\|\xi_0\|_{L^2})$.
This shows the desired integrability conditions and finishes the proof.
\end{proof}

The Lipschitz continuity of the functions $(\tilde{b},\tilde{\sigma})$ implies that
the unique solution $X$ to 
 \eqref{eq:fbsde_fwd} is
 Malliavin differentiable.
This proof naturally  extends \cite[Theorem 2.2.1]{nualart2006}
to \eqref{eq:fbsde_fwd} 
(whose 
initial condition is random and
coefficients 
 are merely integrable in time)
 and hence is omitted.

\begin{Proposition}\l{prop:X_M_derivative}
 Assume the setting in the Theorem \ref{thm:mv-fbsde_regularity}.
Then it holds for all $t\in [0,T]$ that $X_t\in \sD^{1,2}(\sR^n)$, and 
the derivative
$DX=(D X^{(1)},\ldots, D X^{(d)})$,
which is $\sR^{n\t d}$-valued,
satisfies
for $0\le t<s\le T$ that $D_s X_t=0$ and 
for  $0\le s\le t\le T$ that
\begin{align}\l{eq:X_M_derivative}
D_s X_t&=
 \tilde{\sigma} (s,X_s)+
\int_s^t \p \tilde{b}_r D_s X_r \, \d r +
\sum_{k=1}^d
\int_s^t \p \tilde{\sigma}^{(k)}_rD_s X_r\, \d W^{(k)}_{r},
\end{align}
where $\{ \p \tilde{b},(\p \tilde{\sigma}^{(k)})_{k=1}^d\}\subset \cS^\infty(\sR^{n\t n})$
are  
  uniformly bounded by 
  some constant $C$.
\end{Proposition}

We then establish the Malliavin differentiability of the processes $Y,Z$
in Theorem \ref{thm:mv-fbsde_regularity},
which extends \cite[Proposition 5.9]{karoui1997}
to BSDEs with non-differentiable coefficients.

\begin{Proposition}\l{prop:YZ_M_derivative}
 Assume the setting in the Theorem \ref{thm:mv-fbsde_regularity}.
Then 
it holds for a.e.~$t\in [0,T]$ that $(Y_t,Z_t)\in \sD^{1,2}(\sR^m)\t \sD^{1,2}(\sR^{m\t d})$, and 
the derivatives
$DY=(D Y^{(1)},\ldots, D Y^{(d)})$ and $D Z=(D Z^{(1)},\ldots, D Z^{(d)})$,
which are $\sR^{m\t d}$ and  $\sR^{(m\t d)\t d}$-valued, respectively,
 satisfy
for $0\le t<s\le T$ that $D_s Y_t=D_sZ_t=0$ and 
for $0\le s\le t\le T$ that
\begin{align}\l{eq:YZ_M_derivative}
D_s Y^{(j)}_t&=
\p\tilde{g}_TD_s X^{(j)}_T
+\int_t^T
\p{\tilde{f}}_r\cdot D_s \Theta^{(j)}_r \, \d r-
\int_t^T D_s Z^{(j)}_r\,\d W_{r},
\q  j=1,\ldots, d
\end{align}
with 
$\p\tilde{f}_r\cdot D_s\Theta^{(j)}_r
\coloneqq
\p_x\tilde{f}_rD_s X^{(j)}_r
+\p_y\tilde{f}_rD_s Y^{(j)}_r
+
 \p_z\tilde{f}_r D_s Z^{(j)}_r
$,
where 
$DX$ is the Malliavin derivative of $X$,
and 
the random variable
$\p \tilde{g}_T\in L^2(\cF_T;\sR^{m\t n})$ 
and the processes
$\p_x \tilde{f}\in \cS^\infty(\sR^{m\t n})$,
$\p_y \tilde{f}\in \cS^\infty(\sR^{m\t m})$,
$\p_z \tilde{f}\in \cS^\infty(\sR^{m\t (m\t d)})$
are 
  uniformly bounded by some constant $C$.
Moreover, 
it holds for $\d \sP\otimes \d t$ a.e.~that 
$D_tY_t=Z_t$.
\end{Proposition}
\begin{proof}
By using the Lipschitz continuity of    $( \tilde{b},\tilde{\sigma}, \tilde{f}, \tilde{g})$ in their spatial variables, 
we can obtain by using 
the standard mollification argument 
a sequence of coefficients $(\tilde{b}^\eps,\tilde{\sigma}^\eps, \tilde{f}^\eps, \tilde{g}^\eps)_{\eps>0}$
that  
 converge pointwise to
$(\tilde{b},\tilde{\sigma}, \tilde{f}, \tilde{g})$  as $\eps\to 0$,
and 
are smooth  and $L$-Lipschitz continuous with respect to the spatial variables.
For each $\eps>0$, let $\Theta^\eps=(X^\eps,Y^\eps, Z^\eps)$ be the solution to \eqref{eq:fbsde}
with coefficients $(\tilde{b},\tilde{\sigma},  \tilde{f}^\eps, \tilde{g}^\eps)$.
Then we can obtain from  standard stability results  of \eqref{eq:fbsde} that
$(X^\eps,Y^\eps,Z^\eps)\to (X,Y,Z)$ 
as $\eps\to 0$ in $\cS^2\t \cS^2\t \cH^2$
(see e.g.~\cite[Lemma 2.4(ii)]{zhang2004}).

For each $\eps>0$, 
we have $X^\eps_t\in \sD^{1,2}(\sR^{n})$ 
for all $t\in [0,T]$
and the derivative  satisfies
for  $0\le s\le t\le T$ that
\begin{align*}
D_s X^\eps_t&=
 \tilde{\sigma} ^\eps(s,X^\eps_s)+
\int_s^t  \nabla_x\tilde {b}^\eps(r,X^\eps_r) D_s X^\eps_r \, \d r +
\sum_{k=1}^d
\int_s^t  \nabla_x \tilde{\sigma}^{\eps,(k)}(r,X^\eps_r) D_s X^\eps_r\, \d W^{(k)}_{r}.
\end{align*}
Moreover,
by noticing that  the function
\begin{align*}
\sR^m\t \sR^{m\t d}\ni(y,z)\mapsto D_s\tilde{f}^\eps (t,X^\eps_t(\om),y,z)= (\nabla_x\tilde{f}^\eps) (t,X^\eps_t(\om),y,z)D_sX^\eps_t(\om)\in \sR^{m\t d},
\end{align*}
is continuous
for all $t\in [0,T]$ and a.s.~$\om\in \Om$,
and using the boundedness of the function $\nabla_x \tilde{f}^\eps$
and the fact that $\sE\left[\int_0^T\int_0^T|D_sX^\eps_t|^2\,\d t\,\d s\right]<\infty$,
we can 
extend \cite[Proposition 5.3]{karoui1997} and establish 
that $Y^\eps_t\in \sD^{1,2}(\sR^{m})$ 
for all $t\in [0,T]$,
$Z^\eps_t\in \sD^{1,2}(\sR^{m\t d})$ 
for a.e.~$t\in [0,T]$,
and the derivatives  satisfy 
for  $0\le s\le t\le T$, $j=1,\ldots, d$
 that 
\begin{align*}
D_s Y^{\eps,(j)}_t&=
(\nabla\tilde{g}^\eps)(X^\eps_T)D_s X^{\eps,(j)}_T
+\int_t^T
\nabla_{xyz}{\tilde{f}}^\eps(r,\Theta^\eps_r)\cdot D_s \Theta^{\eps,(j)}_r \, \d r-
\int_t^T D_s Z^{\eps,(j)}_r\,\d W_{r}.
\end{align*}
Standard 
moment estimates of  FBSDEs 
(see e.g.~\cite[Theorem 4.4.4]{zhang2017})
show  for all $0\le s\le T$ that 
\begin{align}\l{eq:DYZ_bdd}
\begin{split}
&\sE\bigg[\sup_{s\le r\le T}|D_sY^\eps_r|^2\bigg]+\sE\bigg[\int_s^T|D_sZ^\eps_r|^2\, \d r\bigg]
\\
&\le C
\bigg(
\sE[|(\nabla\tilde{g}^\eps)(X^\eps_T)D_s X^{\eps}_T|^2]
+\sE\bigg[\int_t^T
|\nabla_{x}{\tilde{f}}^\eps(r,\Theta^\eps_r)D_s X^{\eps}_r|^2 \, \d r\bigg]
\bigg)
\\
&\le
C \big( |\tilde{\sigma}(s,0)|^2+ \|\xi_0\|^2_{L^2}+ \|\tilde{b}(\cdot,0)\|^2_{L^1(0,T)}+  \|\tilde{\sigma}(\cdot,0)\|^2_{L^2(0,T)}+1\big),
\end{split}
\end{align}
where for the last inequality we have used
the  estimates of $\sE[\sup_{s\le r\le T}|D_sX^\eps_r|^2]$
and $\|X^\eps\|_{\cS^2}$,
 and
 the fact that 
$|\tilde{b}^\eps(t,0)|\le C( |\tilde{b}(t,0)|+1)$
and $|\tilde{\sigma}^\eps(t,0)|\le C( |\tilde{\sigma}(t,0)|+1)$
for all $t\in [0,T]$,
which follows from  the Lipschitz continuity of  $\tilde{b}(t,\cdot)$ and $\tilde{\sigma}(t,\cdot)$.

We now show the Malliavin differentiability of the processes $Y$ and $Z$.
For each $t\in [0,T]$, 
we have $\lim_{\eps\to 0}\|Y^\eps_t-Y_t\|_{L^2}=0$
and $\sup_{\eps>0}\sE[\int_0^T|D_s Y^\eps_t|^2\, \d s]<\infty$ (see \eqref{eq:DYZ_bdd}),
which together with \cite[Lemma 1.2.3]{nualart2006}
imply that
 $Y_t\in \sD^{1,2}(\sR^{m})$ for all $t\in [0,T]$. 
  To show the differentiability of the process $Z$,  
 we first introduce the random variable 
 $M= \int_0^T Z_r\,\d W_r$.
 The fact that $Z^\eps \to Z$ in $\cH^2$ implies that 
  $M^\eps\coloneqq \int_0^T Z^\eps_r\,\d W_r\to M$ in $L^2(\Om)$.
Moreover, 
for all $\eps>0$,
we can deduce from
the fact that $Z^\eps_t\in \sD^{1,2}(\sR^{m\t d})$ 
for a.e.~$t\in [0,T]$,
the convergence of $(Z^\eps)_{\eps>0}$ in $\cH^2$ and 
the estimate \eqref{eq:DYZ_bdd}  that 
 $M^\eps\in \sD^{1,2}(\sR^m)$ and 
\begin{align*}
&\sE\bigg[\int_0^T |D_sM^\eps|^2\, \d s\bigg]
=\sE\bigg[\int_0^T \left|Z^\eps_s+\int_s^T D_sZ^\eps_r\,\d W_r\right|^2\, \d s\bigg]
\\
&\le C\bigg(\sE\bigg[\int_0^T |Z^\eps_s|^2\, \d s\bigg]
+\sE\bigg[\int_0^T\int_s^T |D_sZ^\eps_r|^2\, \d r\,\d s\bigg]
\bigg)
\le C<\infty
\end{align*}
  with a constant $C$
 uniformly with respect to $\eps$.
This along with \cite[Lemmas 1.2.3]{nualart2006}
shows that
 $M\in \sD^{1,2}(\sR^{m})$,
and hence $Z_t \in \sD^{1,2}(\sR^{m\t d})$ for a.e.~$t\in [0,T]$
(see \cite[Lemma 1.3.4]{nualart2006}).
Therefore,
by using the  Lipschitz continuity of $(\tilde{f}$, $\tilde{g})$ and the differentiability 
of $(X,Y,Z)$,
one can easily deduce 
 the linear FBSDE \eqref{eq:YZ_M_derivative} 
by applying the operator $D$ to \eqref{eq:fbsde_bwd} and the chain rule (see \cite[Proposition 1.2.4]{nualart2006}).
In particular, the random variable  
$\p\tilde{g}_T$ 
and the process 
$\p \tilde{f}$ can be obtained as the weak limits of the sequences
$((\nabla\tilde{g}^\eps)(X_T))_{\eps>0}$ in $L^2(\cF_T;\sR^m)$
and
 $(\nabla_{xyz}{\tilde{f}}^\eps(\cdot,\Theta_\cdot))_{\eps>0}$ in $\cH^2$, respectively,
 which implies the  desired measurability.

 Finally, for each $0\le \theta<t\le T$, we have $Y_t-Y_\theta=-\int_\theta^t\tilde{f}(r,X_r,Y_r,Z_r)\,\d r+\int_\theta^t Z_r\,\d W_r$. Then for all $0\le \theta<s\le t\le T$, 
 we have $D_sY_t=-\int_s^tD_s\tilde{f}(r,X_r,Y_r,Z_r)\,\d r+Z_s+\int_s^t D_sZ_r\,\d W_r$,
 from which  we can conclude that  $D_tY_t=Z_t$ by setting $s=t$.
 \end{proof}

We then give
a more concrete representation of the process $Z$ based on the relation that $D_tY_t=Z_t$,
which is essential for   the regularity estimates of \eqref{eq:mv-fbsde}
 with non-differentiable coefficients;
{see the discussion at the end of this section for details.}
Let us introduce the processes $(\p X,\p Y,\p Z)\in \cS^2(\sR^{n\t n})\t \cS^2(\sR^{m\t n})\t \cH^2(\sR^{(m\t d)\t n})$
satisfying for all $t\in[0,T]$, $j=1,\ldots, n$ that
\begin{align}\label{eq:variation}
\begin{split}
\p X_t&=
 \sI_n+
\int_0^t \p \tilde{b}_r \p X_r \, \d r +
\sum_{k=1}^d
\int_0^t \p \tilde{\sigma}^{(k)}_r \p X_r\, \d W^{(k)}_{r},
\\
\p Y^{(j)}_t&=
\p\tilde{g}_T\p X^{(j)}_T
+\int_t^T
\p \tilde{f}_r\cdot \p \Theta^{(j)}_r \, \d r-
\int_t^T \p Z^{(j)}_r\,\d W_{r},
\end{split}
\end{align}
where 
$ \p \tilde{b},(\p \tilde{\sigma}^{(k)})_{k=1}^d$ are the uniformly bounded  processes  in \eqref{eq:X_M_derivative},
and $\p\tilde{g}_T$ (resp.~$\p \tilde{f}$) is the bounded  $\cF_T$-measurable random variable (resp.~uniformly bounded process)
in \eqref{eq:YZ_M_derivative}.

The next proposition represents the process $Z$ by using 
$\p Y$ and the inverse of $\p X$, which 
 extends 
\cite[Equation (2.13)]{zhang2004} and 
\cite[Equation (3.15)]{lionnet2015} to the present setting 
where the coefficients are non-differentiable in the spatial variables and merely measurable in the time variable.

We emphasize that 
unlike
for the classical FBSDEs,
the processes $(\p X,\p Y,\p Z)$ 
\textit{do not} agree with 
the first variation processes
of solutions $(X,Y,Z)$ to
the MV-FBSDE
 \eqref{eq:mv-fbsde}
(i.e., the derivatives of the solutions  with respect to the initial condition $\xi_0$),
 since the latter ones
also involve the derivatives of 
 marginal distributions of the solutions
with respect to  the initial condition $\xi_0$.

\begin{Proposition}\l{prop:pXYZ}
Assume the setting in Theorem \ref{thm:mv-fbsde_regularity}.
Then we have that:
\begin{enumerate}[(1)]
\item\l{item:pXYZ}
 \eqref{eq:variation} admits a unique solution 
 $(\p X,\p Y,\p Z)\in \cS^2(\sR^{n\t n})\t \cS^2(\sR^{m\t n})\t \cH^2(\sR^{(m\t d)\t n})$
 satisfying for all $p\ge 2$ that
 $\|\p X\|_{\cS^p}+\|\p Y\|_{\cS^p}+\|\p Z\|_{\cH^p}
\le C_{(p)}<\infty$.
\item\l{item:pX_inv}
For all $t\in [0,T]$,
 $\p X_t$ is invertible. 
Moreover, 
for the inverse 
 $(\p X^{-1}_t)_{t\in [0,T]}$ and
for any $p\ge 2$, 
it holds for some constant $C_{(p)} > 0$  that
$\|\p X^{-1}\|_{\cS^p}
\le C_{(p)}$
and
$\|\p X^{-1}_t-\p X^{-1}_s\|_{L^p}\le C_{(p)}|t-s|^{\frac{1}{2}}$
for all $t,s\in [0,T]$.

\item\l{item:Z_representation} \l{item:Z_regularity}
There exists a uniformly bounded process $\p v\in \cS^\infty(\sR^{m\t n})$ 
such that 
it holds   for $\d \sP\otimes \d t$ a.e.~that $Z_t=\p Y_t \, \p X^{-1}_t \, \tilde{\sigma}(t,X_t)=\p v(t) \tilde{\sigma}(t,X_t)$.
\end{enumerate}

\end{Proposition}

\begin{proof}
Due to the boundedness and adaptedness of the coefficients,
it is clear that 
 \eqref{eq:variation} is well-posed
 and admits
the moment bounds 
 in Item (\ref{item:pXYZ}).
We then show Item (\ref{item:pX_inv}) by first introducing  the process 
$M$ as  the solution to 
the following linear SDE:
 \begin{align*}
\begin{split}
 M_t
&=
 \sI_n-
\int_0^t  M_s\bigg[\p \tilde{b}_s -\sum_{k=1}^d \p \tilde{\sigma}^{(k)}_s \p \tilde{\sigma}^{(k)}_s\bigg]\, \d s 
-\sum_{k=1}^d
\int_0^t M_s \p \tilde{\sigma}^{(k)}_s\, \d W^{(k)}_{s}.
\end{split}
\end{align*}
Then  It\^{o}'s formula  shows that $M_t\p X_t=\p X_t M_t= \sI_n$ for all $t\in [0,T]$,
which implies 
for all $t\in [0,T]$
that 
 $\p X_t$ is invertible with the inverse $\p X_t^{-1}=M_t$
 (see  
\cite[p.~126]{nualart2006} for details).
Standard estimates for linear SDEs then lead to 
the desired \textit{a priori} estimates of $\p X^{-1}$. 

Finally, by comparing \eqref{eq:variation} with \eqref{eq:X_M_derivative} and \eqref{eq:YZ_M_derivative}, we can deduce 
from    the uniqueness of solutions to linear FBSDEs that 
$D_s X_t=\p X_t \p X^{-1}_s \tilde{\sigma}(s,X_s)$ 
and $D_sY_t=\p Y_t \, \p X^{-1}_s\, \tilde{\sigma}(s,X_s)$
for all $0\le s\le t\le T$. 
Hence, since it holds for $\d \sP\otimes \d t$ a.e.~that $Z_t=D_t Y_t$,
we can  obtain the first identity in  Item (\ref{item:Z_representation})  by setting $s=t$
in  $D_sY_t=\p Y_t \, \p X^{-1}_s\, \tilde{\sigma}(s,X_s)$.
On the other hand, 
by using the Lipschitz decoupling field $v$ of the process $Y$
and the fact that 
 $X_t\in \sD^{1,2}(\sR^n)$ and $Y_t\in \sD^{1,2}(\sR^{m})$ for all $t\in [0,T]$,
we can deduce from  the chain rule that
$D_sY_t=\p v(t) D_sX_t$. In particular, the process $\p v$ can be obtained as a weak limit of 
$(\nabla_x[v\ast \rho^\eps](\cdot,X_\cdot))_{\eps>0}$ in $\cH^2(\sR^{m\t n})$
with  standard mollifiers   $( \rho^\eps)_{\eps>0}\subset C^\infty(\sR^n)$,
which implies that  $\p v\in \cS^\infty(\sR^{m\t n})$ is uniformly bounded by  some constant $C$.
Therefore, by setting $s=t$ and using the identity that 
$D_s X_t=\p X_t \p X^{-1}_s \tilde{\sigma}(s,X_s)$, we have that $Z_t=D_tY_t=\p v(t) D_tX_t=\p v(t)  \tilde{\sigma}(t,X_t)$,
which completes the proof of the second identity in   Item (\ref{item:Z_representation}).
\end{proof}

With {Proposition} \ref{prop:pXYZ} in hand, we are now ready to prove 
Theorem \ref{thm:mv-fbsde_regularity}.

\begin{proof}[Proof of Theorem \ref{thm:mv-fbsde_regularity}]
We adapt the arguments for 
 \cite[Theorem  3.5]{lionnet2015}
to 
\eqref{eq:fbsde}
 with irregular coefficients,
 and present the main steps for the reader's convenience.
 
Lemma \ref{lemma:lipschitz_coefficients} and standard moment estimates of FBSDE \eqref{eq:fbsde}
(see e.g.\  \cite[Theorems  3.4.3
and 4.4.4]{zhang2017})
give us that 
$\|X\|_{\cS^p}+\|Y\|_{\cS^p}\le  C_{(p)}(1+\|\xi_0\|_{L^{p}})$.
 Moreover, 
Proposition  \ref{prop:pXYZ}, Item (\ref{item:Z_representation})
shows
 for $\d \sP\otimes \d t$ a.e.~that 
 $|Z_t|\le C  | \tilde{\sigma}(t,X_t)|
 =C| {\sigma}(t,X_t,Y_t,\sP_{(X_t,Y_t,Z_t)})|$,
 which together with 
 the assumption of $\sigma$ in 
  (H.\ref{assum:mv-fbsde_lip}) shows that
$\|Z\|_{\cS^p}\le  C_{(p)}(1+\|\xi_0\|_{L^{p}})$.

The H\"{o}lder regularity of the processes $X$ and $Y$
follows directly from 
the fact that $(X,Y,Z)$ solves \eqref{eq:mv-fbsde}
(or equivalently \eqref{eq:fbsde}),
together with
H\"{o}lder's inequality, 
the estimate of $\|(X,Y,Z)\|_{\cS^p}$
and the Burkholder-Davis-Gundy inequality (see e.g.~\cite[Theorem 3.5 (ii)]{lionnet2015}).

Finally, we establish  Item (\ref{item:mv_Z_Holder}) by using the additional assumptions that 
$\sigma$ depends only on the flow $(\sP_{(X_t,Y_t)})_{t\in [0,T]}$
 and satisfies \eqref{eq:sigma_local_holder}.
 Then, 
for any $p\ge 2$ and $s,t\in [0,T]$,
we can obtain from 
  $\cW_2(\sP_{U},\sP_{V})\le \|U-V\|_{L^p}$ 
that 
\begin{align}\l{eq:sigma_tilde_holder}
\begin{split}
&\|\tilde{\sigma}(s,X_{s})-\tilde{\sigma}(t,X_{t})\|_{L^{p}}
\\
&\le \|{\sigma}(s,X_{s},Y_s,\sP_{(X_s,Y_s)})-{\sigma}(t,X_{s},Y_s,\sP_{(X_s,Y_s)})\|_{L^{p}}
\\
&\quad +\|{\sigma}(t,X_{s},Y_s,\sP_{(X_s,Y_s)})-{\sigma}(t,X_{t},Y_t,\sP_{(X_t,Y_t)})\|_{L^{p}}
\\
&\le 
C\{
(1+\|X_s\|_{L^p}+\|Y_s\|_{L^p})
 |s-t|^{1/2}+\|X_{s}-X_{t}\|_{L^p}+\|Y_{s}-Y_{t}\|_{L^p}\},
 \end{split}
 \end{align}
 which along with the estimates in Items (\ref{item:mv_Z_bdd})-(\ref{item:mv_XY_Holder}) gives us that
$\|\tilde{\sigma}(s,X_{s})-\tilde{\sigma}(t,X_{t})\|_{L^{p}}
\le 
C_{(p)}(1+\|\xi_0\|_{L^{p}})|s-t|^{1/2}$
for all $s,t\in [0,T]$.

Now let
 $p\ge 2$ and $\eps>0$
be arbitrary given constants. Let
$p_\eps\coloneqq \frac{p+\eps}{p}>1$ and
$q_\eps\coloneqq \frac{2(p+\eps)}{\eps}>1$,
we have $1/p_\eps+1/q_\eps+1/q_\eps=1$,
which together with H\"{o}lder's inequality shows that 
$\|\zeta \phi \psi\|_{L^1}\le \|\zeta\|_{L^{p_\eps}}\|\phi\|_{L^{q_\eps}}\|\psi\|_{L^{q_\eps}}$.
Let $\pi=\{0=t_0<\cdots<t_N=T\}$ be a partition
with  stepsize $|\pi|=\max_{i=0,\ldots, N-1}(t_{i+1}-t_i)$.
For 
any   given $i\in \{0,\ldots,N-1\}$ and 
 $r\in (t_i,t_{i+1})$,
we can deduce from {Proposition} \ref{prop:pXYZ} Item (\ref{item:Z_regularity}) that   $Z_r-Z_{t_i}=I_{1,r}+I_{2,r}+I_{3,r}$,
where
$I_{1,r}\coloneqq 
\p Y_{r}\big(\p X^{-1}_{r}-\p  X^{-1}_{t_i}\big)\tilde{\sigma}(r,X_{r})$,
$I_{2,r}\coloneqq \p  Y_{r}\,\p  X^{-1}_{t_i}\big[\tilde{\sigma}(r,X_{r})- \tilde{\sigma}(t_i,X_{t_i})\big]$
and
 $I_{3,r}\coloneqq (\p  Y_{r}-\p  Y_{t_i}) \p  X^{-1}_{t_i} \tilde{\sigma}(t_i,X_{t_i})$.
By using
H\"{o}lder's inequality,
Proposition  \ref{prop:pXYZ} Items  \eqref{item:pXYZ}-\eqref{item:pX_inv}
and
 \eqref{eq:sigma_tilde_holder}, we can deduce that
\begin{align}
\|I_{2,r}\|_{L^p}
&\le
\|\p Y_{r}\|_{L^{pq_\eps}}\|\p  X^{-1}_{t_i}\|_{L^{pq_\eps}} \|\tilde{\sigma}(r,X_{r})-\tilde{\sigma}(t_i,X_{t_i})\|_{L^{p+\eps}}
\le C_{(p,\eps)}(1+\|\xi_0\|_{L^{p+\eps}})|\pi|^{\frac{1}{2}}.
\l{eq:I_2}
\end{align}
Then we can proceed along the lines of the proof
of \cite[Theorem 3.5 (iii)]{lionnet2015}
to estimate 
$\|I_{1,r}\|^p_{L^p}$ and 
$\|I_{3,r}\|^p_{L^p}$,
and then establish 
the desired estimate of 
$\sum_{i=0}^{N-1}
\sE[(\int_{t_i}^{t_{i+1}}
|Z_r-Z_{t_{i}}|^2\,\d r)^{p/2}
]\le C_{(p,\eps)}(1+\|\xi_0\|_{L^{p+\eps}})|\pi|^{\frac{1}{2}}$.
The term 
$\sum_{i=0}^{N-1}
\sE[(\int_{t_i}^{t_{i+1}}
|Z_r-Z_{t_{i+1}}|^2\,\d r)^{p/2}
]$ can be estimated by using similar arguments,
which completes the proof of the  estimates in Item (\ref{item:Z_regularity}). 
\end{proof}

We end this section by emphasizing that 
the generalized representation formulas of the process $Z$ in Proposition  \ref{prop:pXYZ}, Item (\ref{item:Z_representation})
are  crucial  
for establishing the regularity estimate 
in Theorem \ref{thm:mv-fbsde_regularity},
especially  the H\"{o}lder regularity of the process $Z$ in Item \eqref{item:Z_regularity}.

Recall that 
  \cite{zhang2004,lionnet2015} 
   established similar regularity results 
for decoupled FBSDE \eqref{eq:fbsde} whose coefficients are $1/2$-H\"{o}lder continuous in $t$, 
uniformly with respect to the spatial variables,
based on a representation  
of the process $Z$ that only holds  when all coefficients
of the FBSDE
are \textit{continuously differentiable} in the spatial variables.
In particular, 
the authors  first 
employ a  mollification argument
to
 construct a sequence of FBSDEs with smooth coefficients $(\tilde{b}^\eps,\tilde{\sigma}^\eps,\tilde{f}^\eps,\tilde{g}^\eps)_{\eps>0}$,
and then 
establish 
a uniform (with respect to $\eps$)
 regularity estimate  for the corresponding solutions $(X^\eps,Y^\eps,Z^\eps)_{\eps>0}$,
 which converge to $(X,Y,Z)$ in $\cS^p\t \cS^p\t \cH^p$ for all $p\ge 2 $ as $\eps\to 0$.
The essential step of the above procedure is to ensure for all $p\ge 2$ that 
$\|\tilde{\sigma}^\eps(t,X^\eps_{t})-\tilde{\sigma}^\eps(s,X^\eps_{s})\|_{L^{p}}
\le K_{p}|t-s|^{\frac{1}{2}}$ for a constant $K_p$ uniformly with respect to $(\eps,t,s)$, 
which holds due to their assumption that $t\mapsto \tilde{\sigma}(t,x)$
is $1/2$-H\"{o}lder continuous
 for all $x\in \sR^n$.

However, this mollification argument fails for the fully-coupled MV-FBSDE \eqref{eq:mv-fbsde}, whose diffusion coefficient is of the form $\sigma(t,X_t,Y_t,\sP_{(X_t,Y_t)})$.
By using the decoupling field $v$ of the process $Y$,
we can rewrite \eqref{eq:mv-fbsde}
as \eqref{eq:fbsde} with 
 a modified diffusion coefficient  
 $\tilde{\sigma}:[0,T]\t \sR^n\to\sR^{n\t d}$,
 which satisfies for all $(t,x)\in [0,T]\t \sR^n$ that 
$\tilde{\sigma}(t,x)=\sigma(t,x,v(t,x),\sP_{(X_t,Y_t)})$
and is  merely measurable  in the time variable.
Hence it is unclear how to mollify the coefficients such that 
$[0,T]\ni t\mapsto \tilde{\sigma}^\eps(t,X^\eps_t)\in L^p(\Om)$
is $1/2$-H\"{o}lder continuous uniformly with respect to $\eps$.

We overcome this difficulty by  extending the representation of the process $Z$ 
to FBSDEs with irregular coefficients that are non-differentiable in the state variables and merely measurable in 
the time variable; see Proposition  \ref{prop:pXYZ}, Item (\ref{item:Z_representation}).
Then, we can establish the regularity of the solution to MV-FBSDEs by directly 
studying the decoupled FBSDE without mollifying the coefficients,
where 
the desired H\"{o}lder continuity of 
$[0,T]\ni t\mapsto \tilde{\sigma}(t,X_t)=\sigma(t,X_t,Y_t,\sP_{(X_t,Y_t)})\in L^p(\Om)$
is inherited from the regularity of the original coefficient $\sigma$
and the processes $(X,Y)$.

\ms
\textbf{Acknowledgements:} W.\ Stockinger is supported by an Upper Austrian Government grant.

\newpage
\appendix

\section{Proof of Corollary \ref{cor:mono}}\l{sec:proof_mono}

In this section, we prove Corollary \ref{cor:mono}
by adapting 
 the method of
continuation in \cite{peng1999,bensoussan2015}
 to the present setting.

We first present  a stability result for the following family of MV-FBSDEs:
for $t\in [0,T]$,
\begin{align}\l{eq:moc}
\begin{split}
\d X_t&=(\lambda b(t,X_t,Y_t,\sP_{(X_t,Y_t,Z_t)})+\cI^b_t)\,\d t +
(\lambda\sigma (t,X_t,Y_t,  \sP_{(X_t,Y_t,Z_t)})+\cI^\sigma_t)\, \d W_t, 
\\
\d Y_t&=-(\lambda f(t,X_t,Y_t,Z_t, \sP_{(X_t,Y_t,Z_t)})+\cI^f_t)\,\d t+Z_t\,\d W_t,
\\
X_0&=\xi,\q Y_T=\lambda g(X_T,\sP_{X_T})+\cI^g_T,
\end{split}
\end{align}
where  $\lambda\in [0,1]$, 
$\xi\in L^2(\cF_0;\sR^n)$,
$(\cI^b,\cI^\sigma,\cI^f)\in \cH^2(\sR^n\t\sR^{n\t d}\t \sR^m)$ and $\cI^g_T\in L^2(\cF_T;\sR^m)$
are given.

\begin{Lemma}\l{lemma:mono_stab}
Suppose
 the functions $(b,\sigma,f,g)$ satisfy the assumptions in 
 Corollary \ref{cor:mono}.
 Then there exists a constant $C>0$ such that,  
 for all $\lambda_0\in [0,1]$,
for every   
${\Theta}\coloneqq(X,Y, Z)\in  \cS^2(\sR^n) \t \cS^2(\sR^m) \t \cH^2(\sR^{m\t d})$
satisfying \eqref{eq:moc}
with 
$\lambda=\lambda_0$,
  functions $(b,\sigma,f,g)$
  and some
$(\cI^b,\cI^\sigma,\cI^f)\in \cH^2(\sR^n\t\sR^{n\t d}\t \sR^m)$,
 $\cI^g_T\in L^2(\cF_T;\sR^m)$,
 $\xi\in L^2(\cF_0;\sR^n)$,
 and for every 
$ \bar{\Theta}\coloneqq(\bar{X},\bar{Y}, \bar{Z})\in  \cS^2(\sR^n) \t \cS^2(\sR^m) \t \cH^2(\sR^{m\t d})$
satisfying  \eqref{eq:moc}
with 
$\lambda=\lambda_0$,
another 4-tuple of functions $(\bar{b},\bar{\sigma},\bar{f},\bar{g})$ 
satisfying merely
 (H.\ref{assum:mv-fbsde_lip}),
 and some 
$(\bar{\cI}^b,\bar{\cI}^\sigma,\bar{\cI}^f)\in \cH^2(\sR^n\t\sR^{n\t d}\t \sR^m)$,
 $\bar{\cI}^g_T\in L^2(\cF_T;\sR^m)$,
 $\bar{\xi}\in L^2(\cF_0;\sR^n)$, we have that 
 \begin{align}\l{eq:mono_stab}
 \begin{split}
 &\|X-\bar{X}\|_{\cS^2}^2+ \|Y-\bar{Y}\|_{\cS^2}^2+ \|Z-\bar{Z}\|_{\cH^2}^2
 \\
& \le C\bigg\{\|\xi-\bar{\xi}\|_{L^2}^2+
\|\lambda_0 (g(\bar{X}_T,\sP_{\bar{X}_T})-\bar{g}(\bar{X}_T,\sP_{\bar{X}_T}))+\cI^g_T-\bar{\cI}^g_T\|_{L^2}^2
\\
&\quad
+\|\lambda_0 (b(\cdot,\bar{X}_\cdot,\bar{Y}_\cdot,\sP_{\bar{\Theta}_\cdot})-\bar{b}(\cdot,\bar{X}_\cdot,\bar{Y}_\cdot,\sP_{\bar{\Theta}_\cdot}))
+\cI^b-\bar{\cI}^b\|_{\cH^2}^2
\\
&\quad
+\|\lambda_0 (\sigma(\cdot,\bar{X}_\cdot,\bar{Y}_\cdot,\sP_{\bar{\Theta}_\cdot})-\bar{\sigma}(\cdot,\bar{X}_\cdot,\bar{Y}_\cdot,\sP_{\bar{\Theta}_\cdot}))
+\cI^\sigma-\bar{\cI}^\sigma\|_{\cH^2}^2
\\
&\quad
+\|\lambda_0 (f(\cdot,\bar{\Theta}_\cdot,\sP_{\bar{\Theta}_\cdot})-\bar{f}(\cdot,\bar{\Theta}_\cdot,\sP_{\bar{\Theta}_\cdot}))
+\cI^f-\bar{\cI}^f\|_{\cH^2}^2
\bigg\}.
\end{split}
  \end{align}

\end{Lemma}

\begin{proof}[Proof of Lemma \ref{lemma:mono_stab}]
Throughout this proof, let 
 $G\in \sR^{m\t n}$ be the matrix in Corollary \ref{cor:mono},
let
$\delta \xi=\xi-\bar{\xi}$,
$\delta \cI^g_T=\cI^g_T-\bar{\cI}^g_T$,
  $g(X_T)=g(X_T,\sP_{X_T})$,
$g(\bar{X}_T)=g(\bar{X}_T,\sP_{\bar{X}_T})$ 
and $\bar{g}(\bar{X}_T)=\bar{g}(\bar{X}_T,\sP_{\bar{X}_T})$,
for each $t\in [0,T]$
 let 
 $\delta\cI^b_t=\cI^b_t-\bar{\cI}^b_t$,
  $\delta\cI^\sigma_t=\cI^\sigma_t-\bar{\cI}^\sigma_t$,
   $\delta\cI^f_t=\cI^f_t-\bar{\cI}^f_t$,
  $b(\Theta_t)=b(t,X_t,Y_t,\sP_{\Theta_t})$,
$b(\bar{\Theta}_t)=b(t,\bar{X}_t,\bar{Y}_t,\sP_{\bar{\Theta}_t})$
and $\bar{b}(\bar{\Theta}_t)=\bar{b}(t,\bar{X}_t,\bar{Y}_t,\sP_{\bar{\Theta}_t})$.
Similarly, we introduce the notation
$\ell(\Theta_t), \ell(\bar{\Theta}_t), \bar{\ell}(\bar{\Theta}_t)$ for $\ell=\sigma, f$ and $t\in [0,T]$.
We also denote by $C$ a generic constant, 
which depends only on the dimensions,
the constant $L$ in (H.\ref{assum:mv-fbsde_lip})  and
the constants $G,\a_1,\b_1,\b_2,L_\phi$ in  Corollary \ref{cor:mono}, and may take a different value at each occurrence.

By applying It\^{o}'s formula to $\la Y_t-\bar{Y}_t, G(X_t-\bar{X}_t)\ra$,  we obtain that 
\begin{align*}
&\sE[\la \lambda_0(g(X_T)-\bar{g}(\bar{X}_T))+\delta I^g_T, G(X_T-\bar{X}_T)\ra]
-\sE[\la Y_0-\bar{Y}_0, G\delta \xi\ra]
\\
&=\sE\bigg[\int_0^T
\la \lambda_0(b(\Theta_t)-\bar{b}(\bar{\Theta}_t))+\delta \cI^b_t, G^*( Y_t-\bar{Y}_t)\ra 
+
\la \lambda_0(\sigma(\Theta_t)-\bar{\sigma}(\bar{\Theta}_t))+\delta \cI^\sigma_t, G^*( Z_t-\bar{Z}_t)\ra
\\
&\quad 
+
\la -\big(\lambda_0(f(\Theta_t)-\bar{f}(\bar{\Theta}_t))+\delta \cI^f_t\big), G( X_t-\bar{X}_t)\ra
\, \d t
\bigg].
\end{align*}
Then, 
by adding and subtracting the terms
$g(\bar{X}_T), b(\bar{\Theta}_t),\sigma(\bar{\Theta}_t),f(\bar{\Theta}_t)$
and applying the monotonicity condition, we can deduce that
\begin{align*}
&
\lambda_0 \a_1\phi_1(X_T,\bar{X}_T)+
\sE[\la \lambda_0(g(\bar{X}_T)-\bar{g}(\bar{X}_T))+\delta I^g_T, G(X_T-\bar{X}_T)\ra]
-\sE[\la Y_0-\bar{Y}_0, G\delta \xi\ra]
\\
&\le 
\sE\bigg[\int_0^T
\la \lambda_0(b(\bar{\Theta}_t)-\bar{b}(\bar{\Theta}_t))+\delta \cI^b_t, G^*( Y_t-\bar{Y}_t)\ra 
+
\la \lambda_0(\sigma(\bar{\Theta}_t)-\bar{\sigma}(\bar{\Theta}_t))+\delta \cI^\sigma_t, G^*( Z_t-\bar{Z}_t)\ra
\\
&\quad 
+
\la -\big(\lambda_0(f(\bar{\Theta}_t)-\bar{f}(\bar{\Theta}_t))+\delta \cI^f_t\big), G( X_t-\bar{X}_t)\ra
\, \d t
\bigg]
\\
&\quad
-
\lambda_0\int_0^T
\bigg(
 \beta_1
 \phi_1(X_t,\bar{X}_t)
+\beta_2
  \phi_2(t,\Theta_t, \bar{\Theta}_t)\bigg)\, \d t,
\end{align*}
 which together with Young's inequality yields for each $\eps>0$ that
\begin{align}\l{eq:mono_stab_decouple}
\begin{split}
&
\lambda_0 \a_1\phi_1(X_T,\bar{X}_T)+
\lambda_0\int_0^T
\bigg(
 \beta_1
 \phi_1(X_t,\bar{X}_t)
+\beta_2
  \phi_2(t,\Theta_t, \bar{\Theta}_t)\bigg)\, \d t
  \\
  &
  \le
\eps( \|X_T-\bar{X}_T\|_{L^2}^2+\|Y_0-\bar{Y}_0\|_{L^2}^2
 +\|\Theta-\bar{\Theta}\|_{\cH^2}^2)
 +C{\eps}^{-1}\textrm{RHS},
 \end{split}
\end{align}
where $\textrm{RHS}$ denotes the right-hand side of \eqref{eq:mono_stab}.

We now separate our discussion into two cases: (1) 
$\b_2>0$ and the estimate \eqref{eq:b_sigma_lip} holds;
(2) 
$\a_1,\b_1>0$ and the estimate  \eqref{eq:f_g_lip} holds.
For the first case, we can obtain from \eqref{eq:mono_stab_decouple}
and $\lambda_0,\a_1,\b_1\ge 0$
 that it holds for all $\eps>0$ that,
\begin{align}\l{eq:mono_stab_decouple_m<n}
\begin{split}
&
\lambda_0\int_0^T
  \phi_2(t,\Theta_t, \bar{\Theta}_t)\, \d t
  \le
\eps( \|X-\bar{X}\|_{\cS^2}^2+\|Y-\bar{Y}\|_{\cS^2}^2
 +\|Z-\bar{Z}\|_{\cH^2}^2)
 +C{\eps}^{-1}\textrm{RHS}.
 \end{split}
\end{align}
Then, by using   the  Burkholder-Davis-Gundy  inequality, \eqref{eq:b_sigma_lip}, Gronwall's inequality
and the fact that $\lambda_0\in [0,1]$, we can 
deduce that 
\begin{align*}
\|X-\bar{X}\|_{\cS^2}^2
&\le 
C\bigg(
\int_0^T \lambda_0\phi_2(t,\Theta_t, \bar{\Theta}_t)\, \d t
+
\|\xi-\bar{\xi}\|_{L^2}^2
\\
&\quad
+\|\lambda_0 (b(\bar{\Theta})-\bar{b}(\bar{\Theta}))
+\delta \cI^b\|_{\cH^2}^2
+\|\lambda_0 (\sigma(\bar{\Theta})-\bar{\sigma}(\bar{\Theta}))
+\delta \cI^\sigma\|_{\cH^2}^2
\bigg),
\end{align*}
which together with \eqref{eq:mono_stab_decouple_m<n} yields for all small enough $\eps>0$ that 
\begin{align*}
\begin{split}
&
\|X-\bar{X}\|_{\cS^2}^2
  \le
\eps(\|Y-\bar{Y}\|_{\cS^2}^2
 +\|Z-\bar{Z}\|_{\cH^2}^2)
 +C{\eps}^{-1}\textrm{RHS}.
 \end{split}
\end{align*}
Moreover,  by standard estimates for MV-BSDEs \eqref{eq:mv-fbsde_bwd}, we can obtain that
\begin{align*}
&\|Y-\bar{Y}\|_{\cS^2}^2
 +\|Z-\bar{Z}\|_{\cH^2}^2
 \\
&  \le
  C\bigg(\|X-\bar{X}\|_{\cS^2}^2
  +
  \|\lambda_0 (g(\bar{X}_T)-\bar{g}(\bar{X}_T))+\delta \cI^g_T\|_{L^2}^2
  +\|\lambda_0 (f(\bar{\Theta})-\bar{f}(\bar{\Theta}))
+\delta\cI^f\|_{\cH^2}^2
\bigg),
\end{align*}
which completes the desired estimate \eqref{eq:mono_stab} for the first case.

For the second case with $\a_1,\b_1>0$, we can obtain from
 \eqref{eq:mono_stab_decouple} that it holds for all $\eps>0$ that,
\begin{align}\l{eq:mono_stab_decouple_m>n}
\begin{split}
&
\lambda_0 \phi_1(X_T,\bar{X}_T)+
\lambda_0\int_0^T
 \phi_1(X_t,\bar{X}_t)\, \d t
    \\
&  \le
\eps( \|X-\bar{X}\|_{\cS^2}^2+\|Y-\bar{Y}\|_{\cS^2}^2
 +\|Z-\bar{Z}\|_{\cH^2}^2)
 +C{\eps}^{-1}\textrm{RHS}.
 \end{split}
\end{align}
Standard stability estimates for MV-BSDEs with Lipschitz coefficients (see e.g.~\cite[Theorem 4.2.3]{zhang2017}) 
shows that
\begin{align*}
&\|Y-\bar{Y}\|_{\cS^2}^2
 +\|Z-\bar{Z}\|_{\cH^2}^2
 \\
&  \le
  C\bigg(
  \|\lambda_0 (g({X}_T)-\bar{g}(\bar{X}_T))+\delta \cI^g_T\|_{L^2}^2
  +\|\lambda_0 
  (
  f(\cdot,X_\cdot,\bar{Y}_\cdot,\bar{Z}_\cdot, \sP_{(X_\cdot,\bar{Y}_\cdot,\bar{Z}_\cdot)})
-\bar{f}(\bar{\Theta}))
+\delta\cI^f\|_{\cH^2}^2
\bigg),
\end{align*}
from which, by using 
\eqref{eq:f_g_lip},
the fact that $\lambda_0\in [0,1]$
and \eqref{eq:mono_stab_decouple_m>n},
we can deduce for all sufficiently small $\eps>0$ that
\begin{align*}
&\|Y-\bar{Y}\|_{\cS^2}^2
 +\|Z-\bar{Z}\|_{\cH^2}^2
\\
&\le
C\bigg(
  \|\lambda_0 (g({X}_T)-{g}(\bar{X}_T))\|_{L^2}^2
 +
  \|\lambda_0 (g(\bar{X}_T)-\bar{g}(\bar{X}_T))+\delta \cI^g_T\|_{L^2}^2
  \\
&\q   +
  \|\lambda_0 
  (
  f(\cdot,X_\cdot,\bar{Y}_\cdot,\bar{Z}_\cdot, \sP_{(X_\cdot,\bar{Y}_\cdot,\bar{Z}_\cdot)})
-f(\bar{\Theta}_\cdot)
)\|_{\cH^2}^2
  +\|\lambda_0 
  (
  f(\bar{\Theta})
-\bar{f}(\bar{\Theta}))
+\delta\cI^f\|_{\cH^2}^2
\bigg)
\\
&  \le
\eps \|X-\bar{X}\|_{\cS^2}^2
 +C{\eps}^{-1}\textrm{RHS}.
\end{align*}
Then,  standard  stability estimates for MV-SDEs with Lipschitz coefficients give that 
\begin{align*}
&\|X-\bar{X}\|_{\cS^2}^2
\\
&\le
C\bigg(
\|\delta \xi\|_{L^2}^2
+\|\lambda_0 
  (
  b(\cdot,\bar{X}_\cdot,{Y}_\cdot,\sP_{(\bar{X}_\cdot,{Y}_\cdot,{Z}_\cdot)})
-\bar{b}(\bar{\Theta}))
+\delta\cI^b\|_{\cH^2}^2
\\
& \q 
+\|\lambda_0 
  (
  \sigma(\cdot,\bar{X}_\cdot,{Y}_\cdot,\sP_{(\bar{X}_\cdot,{Y}_\cdot,{Z}_\cdot)})
-\bar{\sigma}(\bar{\Theta}))
+\delta\cI^\sigma\|_{\cH^2}^2
\bigg)
\\
&  \le
 C\big(\|Y-\bar{Y}\|_{\cS^2}^2+\|Z-\bar{Z}\|_{\cH^2}^2+\|\delta \xi\|_{L^2}^2
 +\|\lambda_0 
  (
  b(\bar{\Theta})
-\bar{b}(\bar{\Theta}))
+\delta\cI^b\|_{\cH^2}^2
\\
&\q
 +\|\lambda_0 
  (
  \sigma(\bar{\Theta})
-\bar{\sigma}(\bar{\Theta}))
+\delta\cI^\sigma\|_{\cH^2}^2
\big)
\le 
 C\cdot \textrm{RHS},
\end{align*}
which completes the proof of  the desired estimate \eqref{eq:mono_stab} for the second case.
\end{proof}

We are now ready to present the proof of Corollary \ref{cor:mono}.

\begin{proof}[Proof of Corollary \ref{cor:mono}]
We shall establish the  well-posedness, stability and \textit{a priori} estimates
 for \eqref{eq:mv-fbsde} with
an initial time
 $t=0$ and initial state $\xi_0\in L^2(\cF_0;\sR^n)$
 by applying Lemma \ref{lemma:mono_stab}. Similar arguments apply to a general 
 initial time 
 $t\in [0,T]$ and initial state $\xi\in L^2(\cF_t;\sR^n)$.

Let us start by proving 
the unique solvability of \eqref{eq:mv-fbsde} with a given $\xi_0\in L^2(\cF_0;\sR^n)$.
To simplify the notation, 
for every $\lambda_0\in [0,1]$, we say $(\cP_{\lambda_0})$ holds if 
 for any 
 $\xi\in  L^2(\cF_0;\sR^n)$,
$(\cI^b,\cI^\sigma,\cI^f)\in \cH^2(\sR^n\t\sR^{n\t d}\t \sR^m)$ and $\cI^g_T\in L^2(\cF_T;\sR^m)$,
\eqref{eq:moc} with $\lambda=\lambda_0$ 
admits a unique solution in $\sB\coloneqq \cS^2(\sR^n) \t \cS^2(\sR^m) \t \cH^2(\sR^{m\t d})$.
It is clear that $(\cP_{0})$ holds since \eqref{eq:moc} is decoupled. 
Now we show 
 there exists a constant $\delta>0$, such that 
 if $(\cP_{\lambda_0})$   holds for some $\lambda_0\in [0,1)$,
 then $(\cP_{\lambda'_0})$ also holds for  all $\lambda_0'\in (\lambda_0,\lambda_0+\delta]\cap [0,1]$.
Note that this claim along with  the method of continuation
 implies the desired unique solvability of \eqref{eq:mv-fbsde} (i.e., \eqref{eq:moc} with $\lambda=1$,
$(\cI^b,\cI^\sigma,\cI^f,\cI^g_T)=0$, $\xi=\xi_0$).

To establish the desired claim, let  $\lambda_0\in [0,1)$
be a constant for which $(\cP_{\lambda_0})$ holds,
$\eta\in [0,1]$ and 
$(\tilde{\cI}^b,\tilde{\cI}^\sigma,\tilde{\cI}^f)\in \cH^2(\sR^n\t\sR^{n\t d}\t \sR^m)$, $\tilde{\cI}^g_T\in L^2(\cF_T;\sR^m)$, 
$\xi\in  L^2(\cF_0;\sR^n)$
be
arbitrarily given coefficients,
we introduce the following mapping $\Xi:\sB\to \sB$
such that 
for all $\Theta=(X,Y,Z)\in \sB$, $\Xi(\Theta)\in \sB$ is the solution to \eqref{eq:moc} 
with $\lambda=\lambda_0$,
$\cI^b_t=\eta b(t,X_t,Y_t,\sP_{\Theta_t})+\tilde{\cI}^b_t$,
$\cI^\sigma_t=\eta \sigma (t,X_t,Y_t,  \sP_{\Theta_t})+\tilde{\cI}^\sigma_t$,
$\cI^f_t=\eta f(t,\Theta_t, \sP_{\Theta_t})+\tilde{\cI}^f_t$
and $\cI^g_T=\eta g(X_T,\sP_{X_T})+\tilde{\cI}^g_T$,
which is well-defined due to the fact that $\lambda_0\in [0,1)$ satisfies the induction hypothesis.
Observe that
by setting $(\bar{b},\bar{\sigma},\bar{f},\bar{g})=(b,\sigma,f,g)$ in Lemma \ref{lemma:mono_stab},
 we see that there exists a constant $C>0$, independent of $\lambda_0$, such that it holds  for all $\Theta,\Theta'\in \sB$ that
 \begin{align*}
 \begin{split}
 &\|\Xi(\Theta)-\Xi({\Theta}')\|_{\sB}^2
 \\
& \le C\bigg\{\|\eta (g({X}_T,\sP_{{X}_T})-{g}({X}'_T,\sP_{{X}'_T}))\|_{L^2}^2
+\|\eta (b(\cdot,{X}_\cdot,{Y}_\cdot,\sP_{{\Theta}_\cdot})-{b}(\cdot,{X}'_\cdot,{Y}'_\cdot,\sP_{{\Theta}'_\cdot}))
\|_{\cH^2}^2
\\
&\quad
+\|\eta (\sigma(\cdot,{X}_\cdot,{Y}_\cdot,\sP_{{\Theta}_\cdot})-{\sigma}(\cdot,{X}'_\cdot,{Y}'_\cdot,\sP_{{\Theta}'_\cdot}))
\|_{\cH^2}^2
+\|\eta (f(\cdot,{\Theta}_\cdot,\sP_{{\Theta}_\cdot})-{f}(\cdot,{\Theta}'_\cdot,\sP_{{\Theta}'_\cdot}))
\|_{\cH^2}^2
\bigg\}
\\
&\le C\eta^2\|\Theta-{\Theta}'\|_{\sB}^2,
\end{split}
  \end{align*}
which shows that $\Xi$ is a contraction when $\eta$ is sufficiently small (independent of $\lambda_0$),
and subsequently leads to  the desired claim due to Banach's fixed point theorem.

For any given $\xi, \xi'\in L^2(\cF_0;\sR^n)$,
the desired stochastic stability of \eqref{eq:mv-fbsde} follows directly from 
 Lemma \ref{lemma:mono_stab}
 by setting 
$\lambda=1$,
$(\bar{b},\bar{\sigma},\bar{f},\bar{g})=(b,\sigma,f,g)$,
$(\bar{\cI}^b,\bar{\cI}^\sigma,\bar{\cI}^f)=(\cI^b,\cI^\sigma,\cI^f)=0$,
 $\bar{\cI}^g_T=\cI^g_T=0$ and
 $\bar{\xi}=\xi'$.
Moreover,  
for any given $\xi\in L^2(\cF_0;\sR^n)$,
 by setting 
$\lambda=1$,
$(\bar{b},\bar{\sigma},\bar{f},\bar{g})=0$ (which  clearly satisfies (H.\ref{assum:mv-fbsde_lip})),
$(\bar{\cI}^b,\bar{\cI}^\sigma,\bar{\cI}^f)=(\cI^b,\cI^\sigma,\cI^f)=0$,
 $\bar{\cI}^g_T=\cI^g_T=0$,
 $\bar{\xi}=0$
 and $(\bar{X},\bar{Y},\bar{Z})=0$
 in  Lemma \ref{lemma:mono_stab},
we can deduce the 
 estimate that
 \begin{align*}
 \begin{split}
 &\|X\|_{\cS^2}^2+ \|Y\|_{\cS^2}^2+ \|Z\|_{\cH^2}^2
 \\
& \le C\bigg\{\|\xi\|_{L^2}^2+
| g(0,\bm{\delta}_{{0}_{n}})|^2
+\| b(\cdot, 0,\bm{\delta}_{{0}_{n+m}})\|_{L^2(0,T)}^2
+\| \sigma(\cdot, 0,\bm{\delta}_{{0}_{n+m}})\|_{L^2(0,T)}^2
\\
&\quad
+\| f(\cdot, 0,\bm{\delta}_{{0}_{n+m+md}})\|_{L^2(0,T)}^2
\bigg\}
\le C(1+\|\xi\|_{L^2}^2),
\end{split}
\end{align*}
which shows the desired moment bound of the processes $(X,Y,Z)$.

Finally, 
for a given initial condition $\xi_0\in L^2(\cF_0;\sR^n)$,
we can conclude 
the desired regularity estimates 
for solutions to \eqref{eq:mv-fbsde} 
from
Theorem \ref{thm:mv-fbsde_regularity}
and
Proposition \ref{eq:decoupling_exist},
which
completes the proof of Corollary \ref{cor:mono}.
\end{proof}

\end{document}